\documentclass[reqno,12pt,letterpaper]{amsart}

\RequirePackage{amsmath,amssymb,amsthm,graphicx,mathrsfs,url}
\RequirePackage[usenames,dvipsnames]{color}
\RequirePackage[colorlinks=true,linkcolor=Red,citecolor=Green]{hyperref}
\RequirePackage{amsxtra}

\relax

\numberwithin{equation}{section}

\def\arXiv#1{\href{http://arxiv.org/abs/#1}{arXiv:#1}}

\newtheorem{theo}{Theorem}
\newtheorem{prop}{Proposition}[section]

\newtheorem{lemm}[prop]{Lemma}

\def\Remark{\noindent\textbf{Remark.}\ }

\let\Im=\Imag % sorry Knuth

\newcommand{\RR}{{\mathbb R}}

\newcommand{\CC}{{\mathbb C}}

\newcommand{\HH}{{\mathbb H}}

\title[Strichartz estimates]%
{Strichartz estimates for convex co-compact hyperbolic surfaces}
\author{Jian Wang}
\email{wangjian@berkeley.edu}
\address{Department of Mathematics, University of California, Berkeley, CA, 94720}

\begin{document}

\begin{abstract}
Using recent work of Bourgain--Dyatlov \cite{BD} we show that for any convex co-compact hyperbolic surface Strichartz estimates for the Schr\"{o}dinger equation hold with an arbitrarily small loss of regularity.
\end{abstract}

\maketitle

\section{introduction}
\label{s: intro}

In a recent paper \cite{BD}, Bourgain--Dyatlov showed that any convex co-compact hyperbolic surface enjoys a resonance free strip with corresponding polynomial bounds on the resolvent. As is well know (see Datchev \cite{Da}) such estimates imply local smoothing with logarithmic loss of regularity. Using the procedure going back to the work of Staffilani-Tataru \cite{ST} we show that this implies Strichartz estimates with an arbitrarily small loss of regularity.

In the case of quotients for which the limit set has dimension $ \delta $ satisfying $ \delta < \frac{ 1 }{ 2 } $, Burq--Guillarmou--Hassel \cite{BGH} showed that these estimates hold without any loss and we suspect that this might be the case in general. However, until \cite{BD} the only estimate valid for surfaces with $ \delta \geq \frac{ 1 }{ 2 }$ was the same as that for compact surfaces, as in the work of Burq--G\'{e}rard--Tzvetkov \cite{BGT}.

In this paper, we always suppose $ M = \Gamma \backslash \mathbb{ H } $ is a convex co-compact surface (for the definition see for instance \cite[Section 2.4]{Bo}). Then we have the following result:

%%%%%%%%%%%%%%%%%%%%%%%%%%%%%%%%%%%%%%%%%%%%%%%%%%%%%%%%%%%%%%%%%%%%%%%%%%%
\begin{theo}
 \label{thm: str}
Suppose $ M = \Gamma \backslash \HH $, and $ u_0 \in C_0^{ \infty } ( M )$, then for $ p , q \geq 2 $ satisfying $ ( p , q ) \neq ( 2 , \infty ) $ and $ \frac{ 1 }{ p } + \frac{ 1 }{ q } = \frac{ 1 }{ 2 } $
\begin{equation}
 \label{ineq: str}
\| e^{ -i t \Delta_M } u_0 \|_{ L^p ( [ 0 , 1 ] , L^q ( M ) ) } \leq C \| u_0 \|_{ H^{ \epsilon } ( M ) }.
\end{equation}
\end{theo}

%%%%%%%%%%%%%%%%%%%%%%%%%%%%%%%%%%%%%%%%%%%%%%%%%%%%%%%%%%%%%%%%%%%%%%%%%%%%%%

We briefly outline the proof. Since $ M $ is a convex co-compact hyperbolic surface, it can be written as a union of a compact set and finitely many half-funnels (\cite[Section 2.4]{Bo}). In the compact part, there exist a (fractal) set of trapped geodesics. From Theorem 2 in \cite{BD} we know that we can bound the cut-off resolvent by $ h^{ -1 } | \log{ h } | $ (see inequality \eqref{eq: res} below), and this enables us to derive Strichartz estimates with only $ \epsilon $-loss of derivatives for arbitrary positive $ \epsilon $. There is no trapping in the half-funnels, hence we have Strichartz estimates for these parts (see \cite[Lemma 2.2]{BGH}). However, since we are only concerned with Strichartz estimates with $ \epsilon $-loss for the whole surface, we only use Strichartz estimates with an $ \epsilon $-loss for these half-funnels. These are obtained by a direct self-contained argument. We remark however that the results of Bouclet \cite{Bo2} give stronger estimates which could be used in case no-loss estimates are obtained in the compact part.

As an application of Theorem~\ref{thm: str} we obtain new local well-poseness results for nonlinear Schr\"{o}dinger equations. Specifically we have the following

%%%%%%%%%%%%%%%%%%%%%%%%%%%%%%%%%%%%%%%%%%%%%%%%%%%%%%%%%%%%%%%%%%%%%%
\begin{theo}
 \label{thm: sch}
Consider the Schr\"{o}dinger equation
\begin{equation}
 i \partial_t u - \Delta u = F ( u ) ,~~~~ u ( 0 , \cdot ) = u_0 ,
\end{equation}
where $ F $ is a nonlinear polynomial of degree $ \beta $ satisfying $ F ( 0 ) = 0 $.
 For any $ s > 1 - \frac{ 2 } { \max{ \{ \beta - 1 , 2 \} } } $, there exists $ p > \beta - 1 $ such that for any $ u_0 \in H^{ s }( M ) $, there exists $ T > 0 $ and a unique solution
$$u \in C ( [ -T , T ] ; H^{ s } ( M ) ) \cap L^p ( [ -T , T ] ; L^{ \infty } ( M ) ). $$
Moreover,

(1) If $ \| u_0 \|_{ H^s ( M ) } $ is bounded, then $ T $ can be bounded from below by a positive constant.

(2) If $ u_0 \in H^r $ for some $ r > s $, then $ u \in C ( [ -T , T ] , H^r ( M ) ) $.
\end{theo}
%%%%%%%%%%%%%%%%%%%%%%%%%%%%%%%%%%%%%%%%%%%%%%%%%%%%%%%%%%%%%%

Burq--G\'{e}rard--Tzvetkov proved a similar result where the same conclusions hold for $ s > 1 - \frac{ 1 }{ \max { \{ \beta - 1 , 2 \} } } $ (see \cite[Proposition 3.1]{BGT}). Thanks to Theorem~\ref{thm: str}, regularity can be lowered. In particular, for cubic non-linearities ($ \beta = 3 $) we have well-poseness in $ H^{ \epsilon } $ for any $ \epsilon > 0 $.
%%%%%%%%%%%%%%%%%%%%%%%%%%%%%%%%%%%%%%%%%%%%%%%%%%%%%%%%%%%%%%%%%%%%%%
\begin{proof}[Proof of Theorem~\ref{thm: sch}.]
We indicate modifications needed in the proof of \cite[Proposition 3.1]{BGT}. Since $ s > 1 - \frac{ 2 }{ \max{ \{ \beta - 1 , 2 \} } } $, we can find $ p > \max{ \{ \beta - 1 , 2 \} } $ such that $ s > 1 - \frac{ 2 }{ p } = \frac{ 2 }{ q } $. Now we choose $ \epsilon > 0 $ satisfying $ s > \frac{ 2 }{ q } + \epsilon $. Let $ \sigma = s - \epsilon > \frac{ 2 }{ q }$. We can define the space $ Y_T $ in the proof of \cite[Proposition 3.1]{BGT} with this new $ \sigma $. Now the rest part of the proof of \cite[Propostion 3.1]{BGT} can be applied without change.
\end{proof}
This paper is organized in the following way: in Section \ref{s: cpt}, we prove Strichartz estimates for the compact region and in Section \ref{s: fun} we deal with estimates in the funnel. A combination of the two gives the estimate for the entire surface.

\medskip\noindent\textbf{Acknowledgements.}
I would like to thank Maciej Zworski for suggesting this problem and for helpful discussions and Jin Long for comments on the first version of this note. Partial support by the National Science Foundation grant DMS-1500852 is also gratefully acknowledged.

%%%%%%%%%%%%%%%%%%%%%%%%%%%%%%%%%%%%%%%%%%%%%%%%%%%%%%%%%
\section{Strichartz estimates for the compact region}
 \label{s: cpt}

We recall from \cite[Section 2.4]{Bo} that a convex co-compact surface $ M $ can be decomposed as $ M = M_0 \cup G_1 \cdots G_N $ where $ \partial M_0 = \bigsqcup_{ j = 1 }^N G_j$ and $ G_j \simeq [ 0, \infty )_r \times ( \RR / k_j \RR )_x $ with the metric $ g|_{ G_j } = d r^2 + \cosh^2 {r} d x^2 $. We refer to $ M_0 $ as the compact part and to $ G_j $ as half funnels. The full funnels are given by $ F_j = \langle z \mapsto k_j z \rangle \backslash \HH $

In this section we prove the Strichartz estimates for the compact region:

%%%%%%%%%%%%%%%%%%%%%%%%%%%%%%%%%%%%%%%%%%%%%%%%%%%%%%%%%%%%%%%
\begin{prop}(Strichartz estimates for the compact region).
 \label{p: str}
Suppose $ M = \Gamma \backslash \mathbb{ H } $, $ \chi \in C_0^{ \infty }( M ) $, and $ u_0 \in C_0^{ \infty }( M ) $. Then for all $ \epsilon > 0 $, and $ p , q \geq 2 $ satisfying $ ( p , q ) \neq ( 2 , \infty ) $ and $ \frac{ 1 }{ p } + \frac{ 1 }{ q } = \frac{ 1 }{ 2 } $, we have
\begin{equation}
 \label{ineq: strcpt}
\| \chi e^{ -i t \Delta_M } u_0 \|_{ L^p ( [ 0 , 1 ] , L^q( M ) ) } \leq C \| u_0 \|_{ H^{ \epsilon }( M ) }.
\end{equation}
for some constant $ C > 0 $.
\end{prop}
%%%%%%%%%%%%%%%%%%%%%%%%%%%%%%%%%%%%%%%%%%%%%%%%%%%%%%%%%%%%%%%%%%%%%%%%%%

We first state Strichartz estimates with logarithmic loss for spectrally localized data:

%%%%%%%%%%%%%%%%%%%%%%%%%%%%%%%%%%%%%%%%%%%%%%%%%%%%%%%%%%%%%%%%%%%%%%%%%%%%%%
\begin{lemm}(Strichartz estimates for spectrally localized data).
 \label{l: specloc}
Suppose $ M = \Gamma \backslash \mathbb{ H } $, $ \varphi \in C_0^{ \infty }( ( \frac{ 1 }{ 2 } , 2 ) ,\mathbb{ R } ) $ and $ \chi \in C_0^{ \infty } ( M ) $. Then for any $ u_0 \in C_0^{ \infty } ( M ) $ we have
\begin{equation}
 \label{ineq: strlog}
\| \chi e^{ -i t \Delta_M } \varphi ( h^2 \Delta_M ) u_0 \|_{ L^p ( [ 0 , 1 ] , L^q ( M ) ) } \leq C | \log{ h } | \| u_0 \|_{ L^2( M ) }.
\end{equation}
\end{lemm}
%%%%%%%%%%%%%%%%%%%%%%%%%%%%%%%%%%%%%%%%%%%%%%%%%%%%%%%%%%%%%%%%%%%%%%%%%%%%

Before proving the lemma we recall the following lemma due to Bouclet (see \cite[Corollary 1.6]{Bo1}).

%%%%%%%%%%%%%%%%%%%%%%%%%%%%%%%%%%%%%%%%%%%%%%%%%%%%%%%%%%%%%%%%%%%%%%%%%%%%
\begin{lemm}
 \label{l: PL}
Suppose $ P $ is an elliptic self-adjoint differential operator of order $ m > 0 $ on $ M = \Gamma \backslash \mathbb{ H } $. If $ \varphi_0 \in C_0^{ \infty }( \mathbb{ R } ) $ and $ \varphi \in C_0^{ \infty }( \mathbb{ R } \setminus \{ 0 \} ) $ satisfy
\begin{equation}
 \label{eq: cutoff}
\varphi_0 ( \lambda ) + \sum_{ k = 1 }^{ \infty } \varphi ( 2^{ -mk } \lambda ) = 1
\end{equation}
for all $ \lambda \in \mathbb{ R } $. Then for $ 2 \leq q < \infty $ and $ f \in C^{ \infty }( M ) $ we have
\begin{equation}
 \label{eq: PL}
\| f \|_{ L^q ( M ) } \leq C ( \| \varphi_0 ( P ) f \|_{ L^q ( M ) } + ( \sum_{ k = 1 }^{ \infty } \| \varphi ( 2^{ -mk } P ) f \|_{ L^q ( M ) }^2 )^{ \frac{ 1 }{ 2 } } ).
\end{equation}
\end{lemm}
%%%%%%%%%%%%%%%%%%%%%%%%%%%%%%%%%%%%%%%%%%%%%%%%%%%%%%%%%%%%%%%%%%%%%%%%%%%%%%

We will also use the following result:
%%%%%%%%%%%%%%%%%%%%%%%%%%%%%%%%%%%%%%%%%%%%%%%%%%%%%%%%%%%%%%%%%%%%%%%%%%%%%%%%%
\begin{lemm}
 \label{l: com}
If $ \chi \in C_0^{ \infty } ( M ) $ and $ \varphi \in C_0^{ \infty } ( \RR ) $ then for $ \chi_1 \in C_0^{ \infty } ( M ) $ satisfying $ \chi_1 = 1 $ on $ { \rm{ supp } } ( \chi ) $ and $ \varphi_1 \in C_0^{ \infty } $ satisfying $ \varphi_1 = 1 $ on ${ \rm{ supp } } ( \varphi ) $ we have
\begin{equation}
\label{eq: com}
[ \chi, \varphi ( - h^2 \Delta ) ] = h A \chi_1 \varphi_1 ( - h^2 \Delta ) + R ( h ),
\end{equation}
where for $ q \geq 2 $, $ A = O_{ L^{q} \to L^q } (1) $ and $ R ( h ) = O_{ H^{ -N } \to H^{ N } } ( h^{ \infty } )$.
\end{lemm}

This Lemma can be deduced from results of Bouclet \cite{Bo3}, but we give a direct argument.

\begin{proof}[Proof of Lemma \ref{l: com}.]
First we show that
\begin{equation}
 \label{eq: loc}
\chi \varphi ( -h^2 \Delta ) = \chi \varphi ( -h^2 \Delta ) \chi_1 + O ( h^{ \infty } )_{ H^{ -N } \to H^{ N } }.
\end{equation}
In fact, by the Helffer-Sj\"ostrand formula (see for instance \cite[Theorem 14.9]{Zw}) we have
\begin{equation}
 \label{eq: HS}
\chi \varphi ( -h^2 \Delta ) ( 1 - \chi_1 ) = \frac{ 1 }{ \pi } \int_{ \CC } \bar{ \partial }_z \tilde { \varphi }( z ) \chi ( -h^2 \Delta - z )^{ -1 } ( 1 - \chi_1 ) dm,
\end{equation}
where $ \tilde \varphi \in C_0^{ \infty } ( \CC ) $ is an almost analytic extension of $ \varphi $.
Now we choose a sequence of cut-off functions $ \{ \chi_j \}_{ j = 2 }^{ \infty } $ such that $ \chi_{ j + 1 }|_{ { \rm { supp } }( \chi_j ) \cup { \rm{ supp } }( \chi )} = 1 $ and $ \chi_1|_{ { \rm{ supp } }( \chi_j ) } = 1 $ for $ j \geq 2 $.
Then for any $ N $
\begin{equation}\begin{split}
\label{eq: adchi}
\chi ( -h^{ 2 } \Delta - z )^{ -1 } ( 1 - \chi_1 ) 
= & \chi \chi_2 \cdots \chi_N ( -h^2 \Delta - z )^{ -1 } ( 1 - \chi_1 ) \\
= & ( -1 )^{ N - 1 } \chi { \rm{ ad } }_{ \chi_2 } \cdots { \rm{ ad } }_{ \chi_N }( -h^2 \Delta - z )^{ -1 } ( 1 - \chi_1 ).
\end{split}\end{equation}
We note that
\begin{equation*}\begin{split}
& { \rm{ ad } }_{ \chi_j }( -h^2 \Delta - z )^{ -1 } := [ ( -h^2 \Delta - z )^{ -1 } , \chi_j ] \\
= & ( -h^2 \Delta - z )^{ -1 }[ \chi_j , -h^2 \Delta - z ]( -h^2 \Delta - z )^{ -1 } 
= O_{ H_h^{ -l } \to H_h^{ -l + 3 } }( h | \Im{ z } |^{ - 2 } )
\end{split}\end{equation*}
for some $ K_l > 0 $.
Hence by iterating we know
\begin{equation}
{ \rm{ ad } }_{ \chi_2 } \cdots { \rm{ ad } }_{ \chi_N }( -h^2 \Delta - z )^{ -1 } = O_{ H_h ^{ -l } \to H_h^{ -l + N + 1 } }( h^{ N - 1 }| \Im{ z } |^{ - N } ),
\end{equation}
for some $ K_{ l , N } > 0 $.
Inserting this in \eqref{eq: adchi} and then \eqref{eq: HS} gives \eqref{eq: loc}.

Equation \eqref{eq: loc} allows us to define the symbol class as in \cite[Definition E1 - E3]{DZ} since now we can work on a compact surface without boundary $ M_0 $ containing $ { \rm{ supp } }( \chi ) \cup { \rm{ supp } }{ \chi_1 } $. In particular, we can use the space $ \Psi^{ -\infty }_{ h }( M_0 ) $.

Now we turn to proving \eqref{eq: com}. We first show that
\begin{equation}
[ \chi , \varphi ( -h^2 \Delta ) ]( 1 - \varphi_1 )( -h^2 \Delta ) = O_{ H^{ -N } \to H^{ N } }( h^{ \infty } ),
\end{equation}
that is,
\begin{equation}
 \label{eq: locphi}
\varphi ( -h^2 \Delta ) \chi ( 1 - \varphi_1 )( -h^2 \Delta ) = O_{ H^{ -N } \to H^{ N } }( h^{ \infty } ).
\end{equation}
We now define $ \varphi_j $ in a similar way to $ \chi_j $ in \eqref{eq: adchi}, then
\begin{equation*}\begin{split}
\varphi ( -h^2 \Delta ) \chi ( 1 - \varphi_1 )( -h^2 \Delta ) 
= & \varphi ( -h^2 \Delta ) \varphi_2 ( -h^2 \Delta ) \cdots \varphi_N ( -h^2 \Delta ) \chi ( 1 - \varphi ) ( -h^2 \Delta ) \\
= & ( -1 )^{ N - 1 } \varphi ( -h^2 \Delta ) { \rm { ad } }_{ \varphi_2 ( -h^2 \Delta ) } \cdots { \rm{ ad } }_{ \varphi_N ( -h^2 \Delta ) } \chi ( 1 - \varphi ) ( -h^2 \Delta ).
\end{split}\end{equation*}
From \cite[Theorem 14.9]{Zw} we know that $ \varphi_j ( -h^2 \Delta ) \in \Psi^{ -\infty }_{ h }( M_0 )$.
Hence $ { \rm{ ad } }_{ { \varphi_j }( -h^2 \Delta ) } \chi = O_{ H^{ -l } \to H^{ l+1 } }( h ) $ and \eqref{eq: locphi} follows.
Now a similar argument to the proof of \eqref{eq: loc} gives \eqref{eq: com}.

We note that $ A = h^{ -1 } [ \chi , \varphi ( -h^2 \Delta ) ] \in \Phi_h^{ -\infty } ( M_0 ) $, therefore $ A = O_{ L^q \to L^q }( 1 ) $ by \cite[Lemma 2.2]{KTZ}.
\end{proof}
%%%%%%%%%%%%%%%%%%%%%%%%%%%%%%%%%%%%%%%%%%%%%%%%%%%%%%%%%%%%%%%%%%%%%%%%%%%%%%%%%%%%

%%%%%%%%%%%%%%%%%%%%%%%%%%%%%%%%%%%%%%%%%%%%%%%%%%%%%%%%%%%%%%%%%%%%%%%%%%%%%%%%%%%%%
\begin{proof}[Proof of Proposition \ref{p: str} assuming Lemma \ref{l: specloc}.]
Let $ \varphi_0 \in C_0^{ \infty }( \RR ) $, $ \varphi \in C_0^{ \infty } ( \RR \setminus \{ 0 \} ) $ such that
\begin{equation}
\varphi_0 ( \lambda ) + \sum_{ k=1 }^{ \infty } \varphi ( 2^{ -2k } \lambda ) = 1
\end{equation}
for all $ \lambda $. Then
\begin{equation}\begin{split}
 \label{eq: prostr}
\| \chi e^{ -i t \Delta } u_0 \|_{ L^p ; L^q } 
\leq & C ( \| \varphi_0 ( \Delta ) \chi e^{ -i t \Delta } u_0 \|_{ L^p ; L^q } + \| ( \sum_{ k = 1 }^{ \infty } \| \varphi ( 2^{ -2 k } \Delta ) \chi e^{ -i t \Delta } u_0 \|_{ L^q }^2 )^{ \frac{ 1 }{ 2 } } \|_{ L^p } ) \\
\leq & C (\| \varphi_0 ( \Delta ) \chi e^{ -i t \Delta } u_0 \|_{ L^p ; L^q } + ( \sum_{ k = 1 }^{ \infty } \| \varphi ( 2^{ -2 k } \Delta ) \chi e^{ -i t \Delta } u_0 \|_{ L^p ; L^q }^2 )^{ \frac{ 1 }{ 2 } } ) \\
= & C ( \| \chi \varphi_0 ( \Delta ) e^{ -i t \Delta } u_0 \|_{ L^p ; L^q } + ( \sum_{ k = 1 }^{ \infty } \| \chi \varphi ( 2^{ -2 k } \Delta ) e^{ -i t \Delta } u_0 \|_{ L^p ; L^q }^2 )^{ \frac{ 1 }{ 2 } } ) \\
& + C ( \| [ \varphi_0 ( \Delta ), \chi ] e^{ -i t \Delta } u_0 \|_{ L^p ; L^q } + ( \sum_{ k = 1 }^{ \infty } \| [ \varphi ( 2^{ -2 k } \Delta ), \chi ] e^{ -i t \Delta } u_0 \|_{ L^p ; L^q }^2 )^{ \frac{ 1 }{ 2 } } ) \\
= : & { \rm{ I } } + { \rm{ II } }.
\end{split}\end{equation}
By Lemma~\ref{l: specloc},
\begin{equation}\begin{split}
{ \rm{ I } } : = & C ( \| \chi \varphi_0 ( \Delta ) e^{ -i t \Delta } u_0 \|_{ L^p ; L^q } + ( \sum_{ k = 1 }^{ \infty } \| \chi \varphi ( 2^{ -2 k } \Delta ) e^{ -i t \Delta } u_0 \|_{ L^p ; L^q }^2 )^{ \frac{ 1 }{ 2 } } ) \\
\leq & C ( \| u_0 \|_{ L^2 } + ( \sum_{ k = 1 }^{ \infty } | \log{ 2^{ -k } } |^2 \| \varphi ( 2^{ -2 k } \Delta ) u_0 \|_{ L^2 } )^{ \frac{ 1 }{ 2 } } ) \\
\leq & C ( \| u_0 \|_{ L^2 } + ( \sum_{ k = 1 }^{ \infty } 2^{ 2k \epsilon } \| \varphi ( 2^{ -2 k } \Delta ) u_0 \|_{ L^2 }^2 )^{ \frac{ 1 }{ 2 } } ) \\
\leq & C \| u_0 \|_{ H^{ \epsilon } }.
\end{split}\end{equation}
For $ { \rm{ II } } $: by \eqref{eq: loc} we know that
\begin{equation}
 \label{eq: II}
\| [ \varphi ( -h^2 \Delta ) , \chi ] e^{ -i t \Delta } u_0 \|_{ L^p ; L^q } \leq h \| A \chi_1 \varphi_1 ( -h^2 \Delta ) e^{ -i t \Delta } u_0 \|_{ L^p ; L^q } + \| R ( h ) e^{ -i t \Delta } u_0 \|_{ L^p ; L^q }.
\end{equation}
By Lemma~\ref{l: com}
\begin{equation}\begin{split}
 \label{A}
\| A \chi_1 \varphi_1 ( -h^2 \Delta ) e^{ -i t \Delta } u_0 \|_{ L^p ; L^q } 
\leq & C \| \chi_1 \varphi_1 ( -h^2 \Delta ) e^{ -i t \Delta } u_0 \|_{ L^p ; L^q } \\
\leq & C | \log{ h } | \| \varphi ( -h^{ 2 } \Delta ) e^{ -i t \Delta } u_0 \|_{ L^2 }.
\end{split}\end{equation}
For the last term in \eqref{eq: II}:
\begin{equation}\begin{split}
 \label{eq: R}
\| R ( h ) e^{ -i t \Delta } u_0 \|_{ L^p ; L^q } 
\leq & C \| R( h ) e^{ -i t \Delta } u_0 \|_{ L^p ; H^{ 1 - 2 / p } } \\
\leq & C h \| e^{ -i t \Delta } u_0 \|_{ L^p ; L^2 } \leq C h \| u_0 \|_{ L^2 }
\end{split}\end{equation}
since $ e^{ -i t \Delta } $ preserves the $ L^2 $ norm.

Finally, we have
\begin{equation}\begin{split}
{ \rm{ II } } : = & C ( \| [ \varphi_0 ( \Delta ), \chi ] e^{ -i t \Delta } u_0 \|_{ L^p ; L^q } + ( \sum_{ k = 1 }^{ \infty } \| [ \varphi ( 2^{ -2 k } \Delta ), \chi ] e^{ -i t \Delta } u_0 \|_{ L^p ; L^q }^2 )^{ \frac{ 1 }{ 2 } } ) \\
\leq & C ( \| u_0 \|_{ L^2 } + ( \sum_{ k = 1 }^{ \infty } | 2^{ -k } \log{ 2^{ -k } } |^2 \| \varphi ( 2^{ -2 k } \Delta ) u_0 \|_{ L^2 }^2 + 2^{ -2 k } \| u_0 \|_{ L^2 }^2 )^{ \frac{ 1 }{ 2 } } )\\
\leq & C ( \| u_0 \|_{ L^2 } + ( \sum_{ k = 1 }^{ \infty } 2^{ 2 k \epsilon } \| \varphi ( 2^{ -2 k } \Delta ) u_0 \|_{ L^2 }^2 )^{ \frac{ 1 }{ 2 } } + ( \sum_{ k = 1 }^{ \infty } 2^{ -2 k } \| u_0 \|_{ L^2 }^2 )^{ \frac{ 1 }{ 2 } } ) \\
\leq & C \| u_0 \|_{ H^{ \epsilon } } .
\end{split}\end{equation}
\end{proof}
%%%%%%%%%%%%%%%%%%%%%%%%%%%%%%%%%%%%%%%%%%%%%%%%%%%%%%%%%%%%%%%%%%%%%%%%%%%%
We now turn to the proof of Lemma \ref{l: specloc}. We need the following lemma.

%%%%%%%%%%%%%%%%%%%%%%%%%%%%%%%%%%%%%%%%%%%%%%%%%%%%%%%%%%%%%%%%%%%%%%%%%%%%
\begin{lemm} (Local smoothing with logarithmic loss).
 \label{l: locsmth}
Suppose $ M = \Gamma \backslash \mathbb{ H } $, $ \chi \in C_0^{ \infty }( M ) $, $ \varphi \in C_0^{ \infty } ( ( \frac{ 1 }{ 2 } , 2 ) ,\mathbb{ R } ) $, and $ u_0 \in C_0^{ \infty } ( M ) $. Then
\begin{equation}
 \label{eq: locsmth}
\| \chi \varphi ( h^2 \Delta_M ) e^{ -i t \Delta_M } u_0 \|_{ L^2 ( [ 0 , 1 ] , L^2 ( M ) ) } \leq C ( h | \log{ h } | )^{ \frac{ 1 }{ 2 } } \| u_0 \|_{ L^2 ( M ) }.
\end{equation}
\end{lemm}
%%%%%%%%%%%%%%%%%%%%%%%%%%%%%%%%%%%%%%%%%%%%%%%%%%%%%%%%%%%%%%%%%%%%%%%%%%%%%
\begin{proof}[Proof of Lemma \ref{l: locsmth}.]
From Theorem 2 in \cite{BD} and the proof of \cite[inequality (6.3.10)]{DZ}, we have the following bound:
\begin{equation}
 \label{eq: res}
\| \chi ( h^2 \Delta_M - ( 1 \pm i \epsilon ) )^{ -1 } \chi \|_{ L^2 ( M ) \rightarrow L^2 ( M ) } \leq C \frac{ | \log { h } | }{ h } ,
\end{equation}
for $ 0 < h < h_0 \ll 1 $ with $ C $ independent of $ h $.
We now use a modification of Kato's argument as presented in \cite[Theorem 7.2]{DZ}. In the notation of that reference, we take $ K ( h ) = \log( 1 / h ) $ to obtain \eqref{eq: locsmth}.
\end{proof}

%%%%%%%%%%%%%%%%%%%%%%%%%%%%%%%%%%%%%%%%%%%%%%%%%%%%%%%%%%%%%%%%%%%%%%%%%%%%
\Remark
From the estimate of the resolvent \eqref{eq: res}, as explained in \cite[Section 7.1]{DZ}, we have the following estimate
\begin{equation}
 \label{eq: Hhalf}
\| \tilde { \chi } e^{ i t \Delta_M } u_0 \|_{ L^{ 2 } ( [ 0 , 1 ] , H^{ \frac{ 1 }{ 2 } } ( M ) ) } \leq C \| u_0 \|_{ H^{ \epsilon } }.
\end{equation}
%%%%%%%%%%%%%%%%%%%%%%%%%%%%%%%%%%%%%%%%%%%%%%%%%%%%%%%%%%%%%%%%%%%%%%%%%%%%%

%%%%%%%%%%%%%%%%%%%%%%%%%%%%%%%%%%%%%%%%%%%%%%%%%%%%%%%%%%%%%%%%%%%%%%%%%%%%%
Now we state a semiclassical dispersive estimate which together with Lemma~\ref{l: KT} gives Strichartz estimates for localized solutions. For the proof of the dispersive estimates, we refer to the proof of Lemma 2.5 in \cite{BGT} and \cite[(4.8)]{KTZ}. Note that though Lemma 2.5 in \cite{BGT} was proved for compact manifolds, the argument applies without change since we are only concern with the compact region.

%%%%%%%%%%%%%%%%%%%%%%%%%%%%%%%%%%%%%%%%%%%%%%%%%%%%%%%%%%%%%%%%%%%%%%%%%%%
\begin{lemm}(Semiclassical dispersion estimate).
\label{l: dis}
Suppose $ M = \Gamma \backslash \mathbb{ H } $, $ \varphi \in C_0^{ \infty } ( \mathbb{ R } ) $, and $ \chi \in C_0^{ \infty } ( M ) $. Then there exists $ \alpha > 0 $, $ C > 0 $, such that for all $ u_0 \in C_0^{ \infty } ( M ) $, $ h \in ( 0 , 1 ] $, we have
\begin{equation}
\label{eq: dis}
\| \chi e^{ -i t \Delta_M } \varphi ( h^2 \Delta_M ) \chi u_0 \|_{ L^{ \infty } ( M ) } \leq \frac{ C }{ | t | + h^{ 2 } } \| u_0 \|_{ L^1 ( M ) }
\end{equation}
for every $t \in [ - \alpha h , \alpha h ] $.
\end{lemm}
%%%%%%%%%%%%%%%%%%%%%%%%%%%%%%%%%%%%%%%%%%%%%%%%%%%%%%%%%%%%%%%%%%%%%%%%%%%

\begin{lemm}(Keel-Tao \cite{KT})
 \label{l: KT}
Let $ ( X , \mathcal{ S } , \mu ) $ be a $ \sigma $-finite measured space, and $ U : \RR \rightarrow B ( L^{ 2 } ( X , \mathcal{ S } , \mu ) )$  be a weakly measurable map satisfying, for some $ C $, $ \sigma > 0 $,
\begin{equation}
\| U ( t ) \|_{ L^2 \rightarrow L^2 } \leq C, \quad t \in \RR,
\end{equation}
and
\begin{equation}
\| U ( t_1 ) U ( t_2 )^* f \|_{ L^{ 1 } \rightarrow L^{ \infty } } \leq \frac{ A }{ | t_1 - t_2 |^{ \sigma } }, \quad t_1 \neq t_2 \in \RR.
\end{equation}
Then for any $ p , q \in [ 1 , \infty ] $ satisfying $ \frac{ 2 }{ p } + \frac{ 2 \sigma } { q } = \sigma $, $ p \geq 2 $ and $ ( p , q ) \neq ( 2 , \infty ) $, we have
\begin{equation}
\| U \|_{ L^2 \rightarrow L^p ; L^q } \leq C^{ \prime },
\end{equation}
for some constant $ C^{ \prime } = C^{ \prime } ( C , \sigma , p , q ) $.
\end{lemm}
%%%%%%%%%%%%%%%%%%%%%%%%%%%%%%%%%%%%%%%%%%%%%%%%%%%%%%%%%%%%%%%%%%%%%%%%%%%%
We will also use the well-known lemma of Christ and Kiselev:
%%%%%%%%%%%%%%%%%%%%%%%%%%%%%%%%%%%%%%%%%%%%%%%%%%%%%%%%%%%%%%%%%%%%%%%%%%%
\begin{lemm}[Christ-Kiselev, \cite{CK}]
 \label{l: CK}
Suppose $ X $ and $ Y $ are Banach spaces and $ K \in C ( B ( X , Y ) ) $, where $ B ( X , Y ) $ is the space of bounded linear mappings from $ X $ to $ Y $. Suppose $ - \infty \leq a < b \leq \infty $. Let
\[
T f ( t ) = \int_a^b K ( t , s) f ( s ) ds, \quad W f = \int_a^t K ( t , s ) f ( s ) ds.
\]
If for $ 1 \leq p < q \leq \infty $
\begin{equation}
\| T \|_{ L^p ( ( a , b ) , X ) \rightarrow L^q ( ( a , b ) , Y ) } \leq C,
\end{equation}
then
\begin{equation}
\| W \|_{ L^p ( ( a , b ) , X ) \rightarrow L^q ( ( a , b ) , Y ) } \leq C^{ \prime },
\end{equation}
for some $ C^{ \prime } = C^{ \prime} ( p , q , C ) $.
\end{lemm}
%%%%%%%%%%%%%%%%%%%%%%%%%%%%%%%%%%%%%%%%%%%%%%%%%%%%%%%%%%%%%%%%%%%%%%%%%%%%%
Now we are in the position to prove Lemma~\ref{l: specloc}. The proof given here is based on the proof of Theorem 3.3 in \cite{BGH}, with some of the ideals also presented in \cite{ST}.
%%%%%%%%%%%%%%%%%%%%%%%%%%%%%%%%%%%%%%%%%%%%%%%%%%%%%%%%%%%%%%%%%%%%%%%%%%%
\begin{proof}[Proof of Lemma~\ref{l: specloc}.]
First of all, form Lemma~\ref{l: dis} and Lemma~\ref{l: KT}, we have
\begin{equation}
 \label{eq: ch}
\| \chi e^{ -i t \Delta_M } \varphi ( h^2 \Delta_M ) u_0 \|_{ L^p ( [ 0 , c h ] , L^q ( M ) ) } \leq C \| u_0 \|_{ L^2 ( M ) }
\end{equation}
for some $ c > 0 $.
By Littlewood-Paley theory, we can assume that $ u_0 $ is localized near frequency $ h^{ - 1 } $ in the sense that $ \varphi ( h^2 \Delta_M ) u_0 = u_0 $. Then we have
\begin{equation}
 \label{eq: locch}
\| \chi e^{ -i t \Delta_M } u_0 \|_{ L^p ( [ 0 , c h ] , L^q ( M ) ) } \leq C \| u_0 \|_{ L^2 ( M ) }.
\end{equation}
Now we choose a time cut-off function $ \psi $ such that $ \psi \in C_0^{ \infty } [ -1 , 1 ] $, $ \psi ( 0 ) = 1 $, and $ \sum_{ j \in \mathbb{ Z } } \psi ( s - j ) = 1 $.
Denote $ u = e^{ -i t \Delta_M } u_0 $, then
\begin{equation}
\chi u = \sum_{ j \in \mathbb{ Z } } \psi( s / h - j ) \chi u = : \sum_{ j \in \mathbb{ Z } } u_j.
\end{equation}
Let $ h = \frac{ 1 }{ N } $, then $ u_0 $ and $ u_N $ can be estimated by \eqref{eq: locch}. For $ 1 \leq j \leq N - 1 $, note
\begin{equation}
( i \partial_t - \Delta_M ) u_j = \frac{ i }{ h } \psi^{ \prime } ( t / h - j ) \chi u - \psi ( t / h - j ) ( \Delta_M \chi u + 2 \nabla \chi \nabla u ) = : w_j.
\end{equation}
By the local smoothing estimate with logarithmic loss we have
\begin{equation}
 \label{eq: l2l2}
\| \tilde{ \chi } u \|_{ L^2 ( [ 0 , 1 ] , L^2 ( M ) ) } \leq C ( h | \log{ h } | )^{ \frac{ 1 }{ 2 } } \| u_0 \|_{ L^2 ( M ) }
\end{equation}
for all $ \tilde{ \chi } \in C_0^{ \infty } ( M ) $. Let $ \tilde{ \chi } \equiv 1 $ on the support of $ \chi $, then
\begin{equation}
 \label{eq: wj}
\sum_{ j } \| w_j \|^2_{ L^2 ( [ 0 , 1 ] , L^2 ( M ) ) } \leq \frac{ 1 }{ h^2 } \| \tilde{ \chi } u \|^2_{ L^2 ; L^2 } \leq C \frac{ | \log{ h } | }{ h } \| u_0 \|^2_{ L^2 }.
\end{equation}
Using Duhamel's formula, we get
\begin{equation}
u_j = -i \int_{ -\infty }^t e^{ -i ( t - s ) \Delta_M } w_j ( s ) ds.
\end{equation}
Let
\begin{equation}\begin{split}
\tilde{ u_j } ( t )
& = -i \int_{ ( j - 1 ) h }^{ ( j + 1 ) h } e^{ -i ( t - s ) \Delta_M } w_j ( s ) ds
= -i e^{ -i t \Delta_M } \int_{ ( j - 1 ) h }^{ ( j + 1 ) h } e^{ i s \Delta_M } w_j ( s ) ds.
\end{split}\end{equation}
Using the dual estimate of \eqref{eq: l2l2} we have
\begin{equation}
\| \int_{ ( j - 1 ) h }^{ ( j + 1 ) h } e^{ i s \Delta_M } w_j ( s ) ds \|_{ L^2 ( M ) } \leq C ( h | \log{ h } | )^{ \frac{ 1 }{ 2 } } \| w_j \|_{ L^2 ; L^2 }.
\end{equation}
Now by \eqref{eq: locch} we get
\begin{equation}
\| \tilde { u_j } \|_{ L^p ; L^q } \leq C ( h | \log{ h } | )^{ \frac{ 1 }{ 2 } } \| w_j \|_{ L^2 ; L^2 }.
\end{equation}
From Lemma~\ref{l: CK} we know
\begin{equation}
\| u_j \|_{ L^p ; L^q } \leq C ( h | \log{ h } | )^{ \frac{ 1 }{ 2 } } \| w_j \|_{ L^2 ; L^2 }.
\end{equation}
Hence
\begin{equation}
\sum_{ j = 1 }^{ N - 1 } \| u_j \|^2_{ L^p ; L^q } \leq C ( h | \log{ h } | ) \sum_{ j = 1 }^{ N- 1 } \| w_j \|^2_{ L^2 ; L^2 } \leq C | \log{ h } |^2 \| u_0 \|^2_{ L^2 }.
\end{equation}
For $ p > 2 $, we have
\begin{equation}
( \sum_{ j = 1 }^{ N - 1 } \| u_j \|^p_{ L^p ; L^q } )^{ \frac{ 2 }{ p }} \leq C | \log{ h } |^2 \| u_0 \|^2_{ L^2 }.
\end{equation}
Finally we get
\begin{equation}
\| \chi u \|_{ L^p ; L^q } \leq C | \log{ h } | \| u_0 \|_{ L^2 }.
\end{equation}
\end{proof}

%%%%%%%%%%%%%%%%%%%%%%%%%%%%%%%%%%%%%%%%%%%%%%%%%%%%%%%%%%%%%%%%%%%%%%%%%%%%%
\section{Strichartz Estimates for the Funnel}
 \label{s: fun}

In this section we will give Strichartz estimates for the funnel. Considering the goal of this paper, we only need the following:
%%%%%%%%%%%%%%%%%%%%%%%%%%%%%%%%%%%%%%%%%%%%%%%%%%%%%%%%%%%%%%%%%%%%%%%%%%%%%
\begin{prop}(Strichartz estimates in the funnel).
 \label{p: noncptstr}
Suppose $ M = \Gamma \backslash \mathbb{ H } $, and \\ $ \chi \in C_0^{ \infty } ( M ) $ such that $ 1 - \chi $ is supported in the half funnels. Then for $ p , q \geq 2 $ satisfying $ ( p , q ) \neq ( 2 , \infty ) $ and $ \frac{ 1 }{ p } + \frac{ 1 }{ q } = \frac{ 1 }{ 2 } $ and any $ \epsilon > 0 $, we have
\begin{equation}
 \label{eq: noncptstr}
\| ( 1 - \chi ) e^{ -i t \Delta_M } u_0 \|_{ L^p ( [ 0 , 1 ] , L^q ( M ) ) } \leq C \| u_0 \|_{ H^{ \epsilon } ( M ) }.
\end{equation}
\end{prop}
%%%%%%%%%%%%%%%%%%%%%%%%%%%%%%%%%%%%%%%%%%%%%%%%%%%%%%%%%%%%%%%%%%%%%%%%%%%%%
The strategy we will follow here is that we will use the cut-off function to restrict the Schr\"{o}dinger equation to the half-funnel where there is no trapping. Since we are dealing with surfaces, the funnel can always be assume to be $ F = \langle z \mapsto k z \rangle \setminus \mathbb{ H } $ for some $ k > 1 $. This will make some of the computations more explicit and direct. The main tools we will use are the following lemma and the Remark of Lemma~\ref{l: locsmth}.
%%%%%%%%%%%%%%%%%%%%%%%%%%%%%%%%%%%%%%%%%%%%%%%%%%%%%%%%%%%%%%%%%%%%%%%%%%%%%
\begin{lemm}
 \label{l: funstr}
Suppose $ F = \langle z \mapsto k z \rangle \setminus \mathbb{ H } $ is a funnel, then for $ u_0 \in C_0^{ \infty } ( F ) $ we have
\begin{equation}
\label{eq: funstr}
\| e^{ -i t \Delta_F } u_0 \|_{ L^p ( [ 0 , 1 ] , L^q ( F ) ) } \leq C \| u_0 \|_{ L^2 ( F ) }.
\end{equation}
\end{lemm}
%%%%%%%%%%%%%%%%%%%%%%%%%%%%%%%%%%%%%%%%%%%%%%%%%%%%%%%%%%%%%%%%%%%%%%%%%%%%
Proposition~\ref{p: noncptstr} and Lemma~\ref{l: funstr} are direct results of \cite[Theorem 1.2 and Theorem 1.3]{Bo2}, but for the reader's convenience we give a self-contained argument here.

We first prove Proposition~\ref{p: noncptstr} using Lemma~\ref{l: funstr} and inequality \eqref{eq: Hhalf}.
%%%%%%%%%%%%%%%%%%%%%%%%%%%%%%%%%%%%%%%%%%%%%%%%%%%%%%%%%%%%%%%%%%%%%%%%%%%%%
\begin{proof}[Proof of Proposition~\ref{p: noncptstr}.]
Without loss of generality, we can assume that $ 1 - \chi $ is supported in a funnel $ F $. Then $ ( i \partial_t - \Delta_M ) u = 0 $ implies
\begin{equation}
( i \partial_t - \Delta_F ) ( 1 - \chi ) u = - [ \Delta_M , \chi ] u.
\end{equation}
Denote $ w = ( 1 - \chi ) u $, then
\begin{equation}\begin{cases}
( i \partial_t - \Delta_F ) w & = - [ \Delta_M , \chi ] u \\
w|_{ t = 0 } & = ( 1 - \chi ) u_0.
\end{cases}\end{equation}
By the Duhamel's formula,
\begin{equation}
w = e^{ -i t \Delta_F } ( 1 - \chi ) u_0 - \int_0^t e^{ -i ( t - s ) \Delta_F } [ \Delta_M , \chi ] u ( s ) ds.
\end{equation}
Denote $ \tilde{ w } = \int_0^1 e^{ -i ( t - s ) \Delta_F } [ \Delta_M , \chi ] u ( s ) ds $. Then by the Christ-Kiselev lemma, we only need to show that
\begin{equation}
 \label{eq: tilw}
\| \tilde{ w } \|_{ L^p ( [ 0 , 1 ] ; L^q ( F ) ) } \leq C \| u_0 \|_{ H^{ 2 \epsilon } ( M ) }.
\end{equation}

Let $ \tilde{ \chi } $ be a cut-off function such that $ \tilde{ \chi } = 1 $ on the support of $ \chi $, then
\begin{equation}\begin{split}
\| \tilde{ \chi } e^{ i t \Delta_M } u_0 \|_{ L^2 ; H^{ \frac{ 1 }{ 2 } + \epsilon } }
\leq & \| [ \tilde{ \chi } , ( I + \Delta )^{ \frac{ \epsilon }{ 2 } } ] e^{ i t \Delta_M } u_0 \|_{ L^2 ; H^{ \frac{ 1 }{ 2 } } }
+ \| \tilde{ \chi } e^{ i t \Delta_M } ( I + \Delta )^{ \frac{ \epsilon }{ 2 } } u_0 \|_{ L^2 ; H^{ \frac{ 1 }{ 2 } } } \\
\leq & \| u_0 \|_{ H^{ 2 \epsilon } }
\end{split}\end{equation}
since $ [ \tilde{ \chi } , ( I + \Delta )^{ \frac{ \epsilon }{ 2 } } ] $ is a differential operator of order $ \epsilon - 1 $. Note that $ [ \Delta_M , \chi ] $ is a first-order differential operator, we find
\begin{equation}
\| [ \Delta_M , \chi ] u \|_{ L^2 ( [ 0 , 1 ] , H_{ { \rm{ comp } } }^{ - \frac{ 1 }{ 2 } + \epsilon }( F ) ) } \leq C \| u_0 \|_{ H^{ 2 \epsilon } ( M ) }.
\end{equation}

Now we define $ T : L^2 ( F ) \rightarrow L^p ( [ 0 , 1 ] , L^q ( F ) ) $, $ u \mapsto e^{ -i t \Delta_F } u $. From Lemma~\ref{l: funstr} we know $ T $ is a bounded operator. Let $ T^* : L^2 ( [ 0 , 1 ] , H_{ { \rm{ comp } } }^{ -\frac{ 1 }{ 2 } + \epsilon } ( F ) ) \rightarrow L^2 ( F ) $, $ w \mapsto \int_0^1 e^{ i s \Delta_F } \chi w ( s ) ds $.
The dual estimate of \eqref{eq: Hhalf} shows that $ T^* $ is bounded as a map from $ L^2 ; H^{ - \frac{ 1 }{ 2 } } $ to $ H^{ -\epsilon } $.
Note that if $ w \in L^2 ( [ 0 , 1 ] , H^{ - \frac{ 1 }{ 2 } + \epsilon }( F ) ) $, then
\begin{equation}\begin{split}
& \| ( I + \Delta_F )^{ \frac{ \epsilon }{ 2 } } \int_0^1 e^{ i t \Delta_F } \chi w ( s ) ds \|_{ L^2 } \\
\leq & \| \int_0^1 e^{ i t \Delta_F } [ ( I + \Delta_F )^{ \frac{ \epsilon }{ 2 } } , \chi ] w ( s ) ds \|_{ L^2 } + \| \int_0^1 e^{ i t \Delta_F } \chi ( I + \Delta_F )^{ \frac{ \epsilon }{ 2 } } w ( s ) \|_{ L^2 } \\
\leq & \| [ ( I + \Delta_F )^{ \frac{ \epsilon }{ 2 } } , \chi ] w \|_{ L^2 ; H^{ -\frac{ 1 }{ 2 } } }
+ \| ( I + \Delta_F )^{ \frac{ \epsilon }{ 2 } } w \|_{ L^2 ; H^{ -\frac{ 1 }{ 2 } } } \\
\leq & \| w \|_{ L^2 ; H^{ - \frac{ 1 }{ 2 } + \epsilon } }.
\end{split}\end{equation}
This indicates $ T^* $ is also a bounded operator from $ L^2 ( [ 0 , 1 ] , H^{ -\frac{ 1 }{ 2 } + \epsilon }( F ) ) $ to $ L^2 $.

Combining the boundedness of these two operators and the fact that
\begin{equation}
\tilde{ w } = T T^* ( [ \Delta_M , \chi ] u ),
\end{equation}
we conclude that \eqref{eq: tilw} is true.
\end{proof}
%%%%%%%%%%%%%%%%%%%%%%%%%%%%%%%%%%%%%%%%%%%%%%%%%%%%%%%%%%%%%%%%%%%%%%%%%%%%
Lemma~\ref{l: funstr} is a special case of \cite[Theorem 1.1]{BGH}, but for the reader's convenience we give a direct proof here.
%%%%%%%%%%%%%%%%%%%%%%%%%%%%%%%%%%%%%%%%%%%%%%%%%%%%%%%%%%%%%%%%%%%%%%%%%%%%
\begin{proof}[Proof of Lemma~\ref{l: funstr}.]
By Theorem 1.1 in \cite{Bo}, the kernel of $ e^{ i t \Delta_{ \mathbb{ H } } } $ is
\begin{equation}
K ( t , z , z^{ \prime } ) = c | t |^{ -\frac{ 3 }{ 2 } } e^{ -\frac{ i t }{ 4 } } \int_{ \rho }^{ \infty } \frac{ e^{ \frac{ i s^2 } { 4 t } } s }{ \sqrt{ \cosh{ s } - \cosh{ \rho } } } ds
\end{equation}
where $ \rho = \rho ( z , z^{ \prime } ) $ is the hyperbolic distance between $ z $ and $ z^{ \prime } $.
For $ 0 < | t | \leq 1 $ we have (see \cite[Proposition 4.2]{Ba})
\begin{equation}
 \label{eq: ker}
| K ( t , z , z^{ \prime } ) | \leq \frac{ C }{ | t | } ( \frac{ \rho }{ \sinh{ \rho } } )^{ \frac{ 1 }{ 2 } }.
\end{equation}
Note $ \Delta_{ z } K = \Delta_{ z^{ \prime } } K = i \partial_t K $, hence for $ 0 < | t | \leq 1 $ and any $ l\geq 0 $, we have
\begin{equation}
| \Delta_z^l K | + | \Delta_{ z^{ \prime } }^l K | \leq | \partial_t^l K | \leq \frac{ C }{ | t |^{ 1 + 2 l } } ( \frac{ \rho }{ \sinh{ \rho } } ).
\end{equation}
This shows that $ K ( t , z , z^{ \prime } ) $ is smooth in $ z $ and $ z^{ \prime } $ when $ t \neq 0 $.
Let $ K_F ( t , z , z^{ \prime } ) : = \sum_{ \gamma \in \Gamma } K ( t , z , \gamma z^{ \prime } ) = \sum_{ n \in \mathbb{ Z } } K ( t , z , k^n z^{ \prime } ) $, then $ K_F $ is the Schwartz kernel of $ e^{ -i t \Delta_F } $. If we can show that for $ 0 < | t | \leq 1 $
\begin{equation}
| K_F ( t , z , z^{ \prime } ) | \leq \frac{ C }{ | t | },
\end{equation}
then by Lemma~\ref{l: KT}, we can get \eqref{eq: funstr}. By \eqref{eq: ker}, we only need to show that
\begin{equation}
\sum_{ n \in \mathbb{ Z } } ( \frac{ \rho_n }{ \sinh{ \rho_n } } )^{ \frac{ 1 }{ 2 } } \leq C
\end{equation}
where $ \rho_n = \rho ( z , k^n z^{ \prime } ) $ is the hyperbolic distance between $ z $ and $ k^n z^{ \prime } $. Note
\begin{equation}
1 + \frac{ ( y - k^n y^{ \prime } )^2 }{ 2 k^n y y^{ \prime } } \leq 1 + \frac{ | z - k^n z^{ \prime } |^2 }{ k^n y y^{ \prime } } = \cosh{ \rho_n } \leq e^{ \rho_n }.
\end{equation}
Hence we have
\begin{equation}
e^{ -\rho_n } \leq \frac { 2 k^n y y^{ \prime } }{ ( y + k^n y^{ \prime } )^2 }.
\end{equation}
Since $ \rho_n \geq 0 $, we have $ \frac{ \rho_n }{ \sinh{ \rho_n } } \leq 4 e^{ - \frac{ \rho_n }{ 2 } } $. Hence
\begin{equation}
 \label{eq: sum}
\sum_{ n \in \mathbb{ Z } } ( \frac{ \rho_n }{ \sinh{ \rho_n } } )^{ \frac{ 1 }{ 2 } } \leq C \sum_{ n \in \mathbb{ Z } } e^{ -\frac{ \rho_n }{ 4 } } \leq C \sum_{ n \in \mathbb{ Z } } ( \frac{ 2 k^n y y^{ \prime } } { ( y + k^n y^{ \prime } )^2} )^{ \frac{ 1 }{ 4 } } = C \sum_{ n \in \mathbb{ Z } } ( \frac{ 2 k^n \lambda }{ ( k^n + \lambda )^2 } )^{ \frac{ 1 }{ 4 } },
\end{equation}
where $ \lambda : = \frac{ y }{ y^{ \prime } } $. Without loss of generality, we can assume $ 1 \leq \lambda \leq k $. Otherwise, since $ y $, $ y^{ \prime } > 0 $, we can find an $ l \in \mathbb{ Z } $ such that $ k^l y^{ \prime } \leq y \leq k^{ l + 1 } y^{ \prime } $ and then we substitute $ y^{ \prime } $ with $ y^{ \prime \prime } : = k^{ l } y^{ \prime } $. Since the sum in \eqref{eq: sum} is taking for all $ n \in \mathbb{ Z } $, we know that this sum will not change and we have $ \lambda^{ \prime } = \frac{ y }{ y^{ \prime \prime } } \in [ 1 , k ] $. Therefore
\begin{equation}
\sum_{ n \in \mathbb{ Z } } ( \frac{ \rho_n }{ \sinh{ \rho_n } } )^{ \frac{ 1 }{ 2 } } \leq C \sum_{ n \in \mathbb{ Z } } ( \frac{ 2 k^{ n + 1 } }{ k^{ 2 n } + 1 } )^{ \frac{ 1 }{ 4 } } \leq C.
\end{equation}
\end{proof}
%%%%%%%%%%%%%%%%%%%%%%%%%%%%%%%%%%%%%%%%%%%%%%%%%%%%%%%%%%%%%%%%%%%%%%%%%

%%%%%%%%%%%%%%%%%%%%%%%%%%%%%%%%%%%%%%%%%%%%%%%%%%%%%%%%%%%%%%%%%%%%%%%%%%%%%%


\begin{thebibliography}{0}

\bibitem[Ba]{Ba} Valeria Banica,
    \emph{The nonlinear Schr\"{o}dinger equation on the hyperbolic space,\/}
    Comm. PDE. \textbf{32}(2007), no. 10, 1643--1677.

\bibitem[BGH]{BGH} Nicolas Burq, Colin Guillarmou and Andrew Hassell,
    \emph{Strichartz estimates without loss on manifolds with hyperbolic trapped geodesices,\/}
    Geometric and Functional Analysis, \textbf{20}(2010), no. 3, 627--656.

\bibitem[BGT]{BGT} Nicolas Burq, Pierre G\'{e}rard and Nikolay Tzvetkov,
    \emph{Strichartz inequalities and the nonlinear Schr\"{o}dinger equation on compact manifolds,\/}
    American Journal of Mathematics, \textbf{126}(2004), no. 3, 569--605.

\bibitem[Bo]{Bo} David Borthwick,
    \emph{Spectral Theory of Infinite-Area Hyperbolic Syrfaces,\/}
    Boston: Birkh\"{a}user, 2007.

\bibitem[BoDy]{BD} Jean Bourgain and Semyon Dyatlov,
	\emph{Spectral gaps without the pressure condition,\/}
	preprint, \arXiv{1612.09040}.

\bibitem[Bou1]{Bo1} Jean-Marc Bouclet,
    \emph{Littlewood-Paley decompositions on manifolds with ends,\/}
    Bull. Soc. Math. Fr., \textbf{138}, fascicule 1 (2010), 1-37.

\bibitem[Bou2]{Bo2} Jean-Marc Bouclet,
    \emph{Strichartz estimates for asymptotically hyperbolic manifolds,\/}
    Analysis and PDE, \textbf{4}(2011), No. 1, 1-84.
\bibitem[Bou3]{Bo3} Jean-Marc Bouclet,
    \emph{Semi-classical calculus on manifolds with ends and weighted $L^p$ estimates,\/}
    Annales de L'Institut Fourier, \textbf{61}(2011), No. 3, 1181-1223.

\bibitem[ChKi]{CK} Michael Christ and Alexander Kiselev,
    \emph{Maximal functions associated to filtrations,\/}
    J. Funct. Anal. \textbf{179}(2001), 409--425.

\bibitem[Da]{Da} Kiril Datchev,
    \emph{Local smoothing for scattering manifolds with hyperbolic trapped sets,\/}
    Comm. Math. Phys. \textbf{286}, no. 3, 837--850.

\bibitem[DyZw]{DZ} Semyon Dyatlov and Maciej Zworski,
    \emph{Mathematical theory of scattering resonances,\/}
    book in progress,
    \url{http://math.mit.edu/~dyatlov/res/res_20170323.pdf}.

\bibitem[KeTa]{KT} Markus Keel and Terence Tao,
    \emph{Endpoint Strichartz estimates,\/}
    Amer. J. Math. \textbf{15}(1998), 955--980.

\bibitem[KTZ]{KTZ} Herbert Koch, Daniel Tataru and Maciej Zworski,
    \emph{Semiclassical $L^p$ Estimates,\/}
    Annales Henri Poincar\'e, 8, Number 5 (2007), 885--916.

\bibitem[StTa]{ST} Gigliola Staffilani and Daniel Tataru,
	\emph{Strichartz estimates for a Schr\"{o}dinger operator with nonsmooth coefficients,\/}
	Comm. Part. Diff. Eq. \textbf{27}(2002), no. 7-8, 1337--1372.
\bibitem[Zw]{Zw} Maciej Zworski,
    \emph{Semiclassical Analysis,\/}
    Vol. 138. Providence, RI: American Mathematical Society, 2012.



\end{thebibliography}
\end{document}